\documentclass[12pt]{article}
\usepackage{bbm}
\usepackage{graphicx}
\graphicspath{ {./images/} }
\usepackage{wrapfig}
\usepackage[mathscr]{euscript}
\usepackage{subcaption}
\usepackage{hyperref}
\usepackage{epstopdf}
\usepackage{enumerate}
\usepackage{longtable}
 \usepackage{float}
\usepackage{amsmath,amsfonts,amssymb,amsthm,epsfig,epstopdf,titling,url,array,multirow}
\usepackage{color}
\usepackage{framed}
\usepackage[utf8]{inputenc}
\usepackage[english]{babel}
\usepackage{xparse}
\usepackage{authblk}
\usepackage[margin=2.5cm]{geometry}

 \NewDocumentCommand{\INTERVALINNARDS}{ m m }{
 	#1 {,} #2
 }
 \NewDocumentCommand{\interval}{ s m >{\SplitArgument{1}{,}}m m o }
 {
 	\IfBooleanTF{#1}{
 		\left#2 \INTERVALINNARDS #3 \right#4
 	}{
 		\IfValueTF{#5}{
 			#5{#2} \INTERVALINNARDS #3 #5{#4}
 		}{
 			#2 \INTERVALINNARDS #3 #4
 		}
 	}
 }

\newtheorem{thm}{Theorem}[section]
\newtheorem{maintheorem}{Theorem}

\newtheorem{cor}{Corollary}[section]
\newtheorem{lem}{Lemma}[section]
\newtheorem{prop}{Proposition}[section]

\newtheorem{rem}{Remark}[section]
\newtheorem{obs}[thm]{Observation}

\numberwithin{equation}{section}

\date{}
\begin{document}
\title{Newton's method applied to rational functions: Fixed points and Julia sets}
\author[1]{Tarakanta Nayak 
    \footnote{tnayak@iitbbs.ac.in}}
\author[2]{Soumen Pal \footnote{soumen.pal.new@gmail.com}}
\author[1]{Pooja Phogat  \footnote{Corresponding author, poojaphogat174acad@gmail.com}}
\affil[1]{Department of Mathematics, 
		Indian Institute of Technology Bhubaneswar, India}
\affil[2]{Department of Mathematics, 
		Indian Institute of Technology Madras, India}
	\date{}
\maketitle
\begin{abstract}
For a rational function  $R$, let  
$N_R(z)=z-\frac{R(z)}{R'(z)}.$ Any such $N_R$ is referred to as a  Newton map. We determine all the rational functions $R$ for which $N_R$ has exactly two attracting fixed points, one of which is an exceptional point. Further, if all the repelling fixed points of any such Newton map are with multiplier $2$, or the multiplier of the non-exceptional attracting fixed point is at most $\frac{4}{5}$, then its Julia set is shown to be connected.  
If a polynomial $p$ has exactly two roots, is unicritical but not a monomial, or $p(z)=z(z^n+a)$ for some $a \in \mathbb{C}$ and $n \geq 1$, then we have proved that the Julia set of $N_{1/p}$ is totally disconnected. For the McMullen map $f_{\lambda}(z)=z^m - \frac{\lambda}{z^n}$, $\lambda \in \mathbb{C}\setminus \{0\}$ and $m,n \geq 1$, we have proved that the Julia set of $N_{f_\lambda}$ is connected and is invariant under rotations about the origin of order $m+n$. All the connected Julia sets mentioned above are found to be locally connected.
\end{abstract}

\section{Introduction}
The Newton's method  applied to a polynomial $p: \widehat{\mathbb{C}} \to  \widehat{\mathbb{C}}$ is defined by $N_p(z)=z-\frac{p(z)}{p'(z)}.$
It is a well-known root-finding method that has been extensively studied (see, for example, \cite{DrachSchleicher2022,HSS2001,przy-1989,lei}). The Newton's method applied to $f = pe^q$, where $p$ and $q$ are polynomials,  is a rational map and has also been explored (see \cite{Haruta1999,Mamayusupov2019,scl}). This is a generalization in the sense that $f$ is a polynomial whenever $q$ is constant. This article deals with the Newton's method applied to rational functions, which is a generalization in a different direction.
\par 
 The Julia set of a rational function $R$, denoted by $\mathcal{J}(R)$,  is the set of all points in a neighbourhood of which the family of functions $\{R^n\}_{n \geq 0}$ is not equicontinuous. The Fatou set of $R$, denoted by $\mathcal{F}(R)$, is the complement of $\mathcal{J}(R)$ in $\widehat{\mathbb{C}}$.  A point $z_0 \in \widehat{\mathbb{C}}$ is called a fixed point of $R$ if $R(z_0)=z_0$. For $z_0= \infty$, we use the local chart $w=\frac{1}{z}$. The multiplier of $z_0$ as a fixed point of $R$ is defined as $\lambda :=\lambda_{z_0}=R'(z_0)$.
 The fixed point $z_0$ is called repelling, indifferent, or attracting if $|\lambda|>1,$ $|\lambda| =1$ or $|\lambda|<1$, respectively. An attracting fixed point with multiplier $0$ is called superattracting. A  repelling or indifferent fixed point with multiplier $1$  is called weakly repelling. Each rational function with degree at least two has at least one weakly repelling fixed point (see Corollary 12.7, \cite{Milnor_book}). There is a beautiful result by Shishikura connecting the Julia set and the fixed points of a rational function.
 \begin{thm}[\cite{Shishikura2009}]
If the Julia set $\mathcal{J}(R)$ of a rational function $R$ with degree at least two is disconnected, then there exist at least two weakly repelling fixed points lying on two different components of the Julia set. In particular, if $R$ has exactly one repelling fixed point, then $\mathcal{J}(R)$ is connected.
 	\label{shishikura}
 \end{thm}
 For the Newton map $N_{p e^q}$, the point at $\infty$ is either a repelling fixed point (whenever $q$ is constant) or an indifferent fixed point with multiplier $1$ (whenever $q$ is non-constant). All other fixed points are attracting. A proof of this fact can be found in  Proposition 1, \cite{Haruta1999}, as well as in Proposition 2.11, \cite{scl}. It follows from Theorem \ref{shishikura} that the Julia set $\mathcal{J}(N_{p e^q})$ is connected. The situation is very much different for the Newton's method applied to a  non-polynomial rational function. To proceed with the discussion, let $R$ be a rational function and, 
 \begin{equation*}\label{formula_N_R}
 N_R(z)=z-\frac{R(z)}{R'(z)}.
 \end{equation*}
Throughout this article, we refer to a rational map as a Newton map if it is $N_R$ for some rational function $R$, where a polynomial is also considered as a rational function. A Newton map is called quadratic or cubic if its degree (not the degree of $R$ from which it arises) is two or three, respectively. Further, the degree of a Newton map is taken to be at least two unless stated otherwise. 
\par Each pole of $R$ is a repelling fixed point of $N_R$ (see Lemma~\ref{NFP}). If $N_R$ has at least two repelling fixed points, then it is not straightforward any more to determine the connectivity of $\mathcal{J}(N_R)$ using 
 Theorem~\ref{shishikura}. Indeed, one needs to know whether there are two weakly repelling fixed points lying on two different components of the Julia set or not. This article primarily deals with the connectedness of the Julia set of $N_R$ for various types of $R$.
\par 
The notion of conjugacy is required to proceed with our discussion. Two rational functions $R_1$ and $R_2$ are said to be (conformally) conjugate if there is a M\"{o}bius map $\phi$ such that $\phi^{-1} \circ R_1 \circ \phi =R_2$. In this case, we also say that $R_1$ is $R_2$ up to conjugacy. A point $z_0$ is a fixed point of $R_2$ if and only if $\phi(z_0)$ is a fixed point of $R_1$. More importantly, the multiplier of $z_0$ is the same as that of $\phi(z_0)$. Further, the Julia set of $R_1$ is the $\phi$-image of the Julia set of $R_2$. As $\phi$ takes connected sets to connected sets, the Julia set of $R_1$ is connected if and only if the Julia set of $R_2$ is connected.
\par  
The study of Newton maps $N_R$, where $R$ is a non-polynomial rational function, was initiated by Barnard et al., who proved that if a quadratic Newton map is conjugate to a polynomial, then its Julia set is connected (see Corollary 3.4, \cite{Barnard2019}). Note that  there is a rational function $R$ with any given degree  such that $N_R$ is quadratic (see Theorem 3.1,~\cite{Nayak_Pal2022}). Later, Nayak and Pal considered all quadratic Newton maps (even if these are not conjugate to any polynomial) and proved that there are only two quadratic Newton maps up to conjugacy. More precisely, each quadratic Newton map is conformally conjugate to $$N_1 (z)=\frac{(d_1 + d_2 -1)z^2 +(1- d_1)z}{(d_1+d_2)z-d_1}~~ \text{ or }~~ N_2 (z)=\frac{(e_1 + e_2 +1)z^2 +(-1- e_1)z}{(e_1 +e_2)z-e_1}$$ for some positive integers $d_1, d_2, e_1,e_2$ (see Remark 3.2, \cite{Nayak_Pal2022}). The Julia set of $N_1$ is a Jordan curve, whereas that of $N_2$ is totally disconnected (see Theorem 1.1, \cite{Nayak_Pal2022}). They also proved that if a cubic Newton map is conjugate to a polynomial, then its Julia set is connected (see Theorem 1.2, \cite{Nayak_Pal2022}).
\par 
First, we look for Newton maps of arbitrary degree with connected Julia sets. For every polynomial $q$, $N_q$ has exactly one repelling fixed point. If $p$ is a monic polynomial with $p(0)\neq 0$ and $\deg(p) \leq k+1$ for some $k \geq 1$, then $N_{p(z)/z^k}$ has exactly one repelling fixed point, namely $0$. This follows from Lemma~\ref{NFP}. Proposition~\ref{conj_to_poly} proves that every Newton map with exactly one repelling fixed point is actually conjugate to $N_q$ for some polynomial $q$. This gives a class of non-polynomial rational functions whose Newton's method have connected Julia sets. We consider Newton maps with two attracting fixed points, one of which is exceptional.  
For a rational function $R$, a point $ w_0 \in \widehat{\mathbb{C}} $ is said to be an exceptional point if its backward orbit $ \{ z : R^n(z) = w_0 \text{ for some positive integer } n \}$ contains at most two elements. At the other extreme are Newton maps with a single attracting fixed point. This is a necessary condition for a totally disconnected Julia set. 

\par A Newton map with an exceptional point is conjugate to a polynomial. All quadratic and cubic Newton maps that are conjugate to some polynomial are already determined in Theorem 3.4 and Table 1,~\cite{Nayak_Pal2022} respectively.    
There can be at most two exceptional points. We have shown in Proposition~\ref{quadratic-newton} that if a Newton map has two exceptional points, then it is  $z^2$ up to conjugacy. If a Newton map has a single exceptional point, it must be a superattracting fixed point. It now follows from Corollary \ref{ext_one_attr}, proved in Section 2, that the Newton map has at least one more attracting fixed point whenever it has an exceptional point.
The following theorem provides a necessary condition for a Newton map to have exactly two attracting fixed points, one of which is exceptional.

\begin{maintheorem}
If a Newton map has exactly two attracting fixed points, one of which is an exceptional point then it is  conjugate to $N_R$, where  $R(z)=\frac{z^d}{p(z)},$ for some $d \geq 1$ and some monic polynomial $p$   with  $p(0) \neq 0$ and $\operatorname{deg}(p)=d$. Moreover, we have the following up to conjugacy.
	\begin{enumerate}
		\item If $p$ is generic, then there is exactly one Newton map, namely  $\frac{z^{d+1} +(d-1)z}{d}$.
		\item If $d=3,4$ or $5$, then there are exactly three, five or eight Newton maps, respectively.
  \end{enumerate}
	\label{two-attracting-fixedponts}
\end{maintheorem}
As mentioned earlier, to have a totally disconnected Julia set, a Newton map must have a single attracting fixed point. Using conjugacy, we can consider this attracting fixed point to be  infinity. Hence, $R$ is of the form $R(z)=1/p(z)$, where $p$ has at least two distinct roots. For certain classes of polynomials, we are able to prove that the Julia set of $N_{1/p}$ is totally disconnected. These Newton maps can have any prescribed number of repelling fixed points unless $p$ has exactly two roots. 

 \begin{maintheorem}\label{diconnected_gen}
Let $p$ be a polynomial of degree at least two. Then, the Julia set $\mathcal{J}(N_{\frac{1}{p}})$ is totally disconnected whenever any one of the following is true.
\begin{enumerate}
  \item $p$ has exactly two roots.
  \item $p$ is unicritical and its critical point is not a root.
  \item $p(z)=z(z^n+a)$, where $a\in \mathbb{C}\setminus \{0\}$ and $n \geq 1$.
\end{enumerate}
\end{maintheorem}
All the Newton maps mentioned in Theorem~\ref{two-attracting-fixedponts} can be described in terms of the multipliers of their fixed points, and their Julia sets are shown to be connected. This is our next result.
\begin{maintheorem}
	Let a Newton map have exactly two attracting fixed points, one of which is an exceptional point. If all the repelling fixed points are with multiplier $2$, or the multiplier of the non-exceptional attracting fixed point is at most $\frac{4}{5}$, then the Julia set of the Newton map is connected.
	\label{connected-JS}
\end{maintheorem}
Each finite simple pole of $R$ is a repelling fixed point of $N_R$ with multiplier $2$. However, the point at $\infty$ can also be such a fixed point for $N_R$, in particular when the degree of the numerator of $R$ is exactly two more than that of the denominator (see Lemma~\ref{NFP}).
\par 
 As discussed earlier, if a Newton map has exactly one repelling fixed point, then its Julia set is connected. We consider the situation when there are exactly two repelling fixed points. If $R(z) = \frac{p(z)}{z^d}$, for a non-monomial polynomial $p$ with $\deg(p) > d + 1$, then $N_R$ has exactly two repelling fixed points, namely $0$ and $\infty$ (see Lemma~\ref{NFP}). For instance, when $p(z) = z^m - \lambda$ with $m \geq 2$ and $\lambda \neq 0$, $R$ belongs to one of the well-known classes of rational functions, namely the McMullen maps $f_\lambda(z)=z^m-\frac{\lambda}{z^n}$, where $m,n  \geq 1$ and $\lambda \neq 0$. 
 This family of maps was first introduced by McMullen in 1988 (\cite{McMullen1998}) for the specific case $m=2, n=3$. Later, in 2005, Devaney et al. studied the general case (see \cite{DevaneyLook2005}, \cite{Devaney2005}) followed by a series of papers by several authors (for example, see~\cite{QWY2012}). We obtain the following result.
\begin{maintheorem}\label{New_Mc}
Let $f_\lambda (z)=z^m-\frac{\lambda}{z^n}, m,n  \geq 1,   \lambda \neq 0$. Then the Julia set of $N_{f_{\lambda}}$ is connected.
\end{maintheorem}
The Julia set of a rational function $R$ is often invariant under some holomorphic Euclidean isometries of the plane. The collection of all such isometries is known as the symmetry group of the Julia set and is denoted by $\Sigma R$. In other words,
$$\Sigma R=\{\sigma(z)=\mu z +\alpha: |\mu|=1 \text{ and } \sigma(\mathcal{J}(R))=\mathcal{J}(R)\}.$$
It is important to note that $\sigma(\mathcal{F}(R)) = \mathcal{F}(R) $  whenever $\sigma \in \Sigma R$. Further, if a Fatou component $U$, i.e., a maximally connected subset of the Fatou set, contains the origin, then $\sigma (U)=U$ for each $\sigma \in \Sigma R$. The symmetry group of $N_{f_\lambda}$ is determined.
\begin{maintheorem}\label{McM_sym}
For $f_\lambda(z)=z^m-\frac{\lambda}{z^n}$,	if $m+n>2$, then	$ \Sigma N_{f_\lambda}=\left\{z \mapsto \mu z: \mu^{m+n}=1\right\}$.
\end{maintheorem}
  If $m=n=1$ is taken in Theorem~\ref{McM_sym}, then the Julia set of $N_{f_\lambda}$ is a line (see Remark~\ref{newton-line}(2)).
  \par 
  The structure of this article is as follows. Section~\ref{Sec2_Prop_N_map} contains some basic properties of the Newton maps. In Section~\ref{Sec3_conf_equiv}, we prove Proposition~\ref{conj_to_poly} and Theorem \ref{two-attracting-fixedponts}. Section~\ref{Sec4_totally_discon} is dedicated to the Newton maps with totally disconnected Julia sets and contains the proof of Theorem \ref{diconnected_gen}. Section~\ref{Sec5_Conn_J_set} deals with the Julia set of the Newton maps mentioned in Theorem~\ref{two-attracting-fixedponts} and the Newton's method applied to the McMullen maps. This section contains the proofs of Theorems \ref{connected-JS}, \ref{New_Mc}, and \ref{McM_sym}.
  
\section{ Properties of Newton maps}\label{Sec2_Prop_N_map}
This section contains some useful properties of Newton maps. A useful fact is that two different rational functions may lead to the same Newton map up to conjugacy.
\begin{lem}[Scaling property] Let $a, b, \lambda \in \mathbb{C}$ with $a, \lambda \neq 0$. If $T(z)=az+b$ and $R$ is a rational function such that $S(z)=\lambda R(T(z))$, then $T \operatorname{o} N_S\operatorname{o}T^{-1}=N_R$.\label{scaling}
\end{lem}
The above lemma follows from Lemma 8, which is proved in ~\cite{BH2003}.
\par 
Let $R(z)=\frac{p(z)}{q(z)}$ be a rational function where $p$ and $q$ are polynomials without any common factor and with respective degrees $d$ and $e$. If $m$ and $n$ are the numbers of distinct roots and poles of $R$, respectively, then the degree of $N_R$ is given by the following formula (page 4, \cite{Nayak_Pal2022}).
\begin{equation}\label{deg_N_R}
deg(N_R)=\begin{cases}
m+n-1  &\mbox{if} \hspace{0.3 cm}  d=e+1 \\
m+n  &\mbox{if} \hspace{0.3 cm}   d\neq e+1. 
\end{cases}
\end{equation}
There are some almost trivial observations when $R$ has a single pole, a single root, or $R$ is a M\"{o}bius map.
\begin{rem}
\begin{enumerate} Using the Scaling property (see Lemma~\ref{scaling}), we have the following.
\item If $R(z) =\frac{c}{(z-z_0)^k}$ for some $c \neq 0,z_0 \in \mathbb{C}$ and $ k \geq 1$, then $N_R $ is conjugate to $N_{\frac{R(z+z_0)}{c}}$. This is nothing but $N_{\frac{1}{z^k}}(z)=(1+\frac{1}{k})z$. If $R(z)=c  (z-z_0)^k$ for $k \geq 1$, then it can be seen similarly that $N_R$ is conjugate to $(1-\frac{1}{k})z$.
\item Let $R(z)=\frac{az+b}{cz+d}$, where $ad-bc\neq 0$. If $c=0$, then $N_R$ is a constant map. Let $c\neq 0$.
If $a=0$, then $N_R$ is conjugate to  $N_{\frac{1}{z}}$, which is $2z$. If $a\neq 0$, then  $N_R$ is conjugate to $N_{\frac{z}{z-1}}$, i.e.,  $z^2$. However, there are rational functions $R$ with degree at least two such that $N_R (z)$ is conjugate to $z^2$. For example,  $N_{\frac{ z^2-1}{z}}(z)=\frac{2z}{z^2 +1}$ and $N_{z^2-1}(z)=\frac{z^2 +1}{2z}$.
	\end{enumerate}
	\label{NM-linear}
	\end{rem}
The nature of all the fixed points of a Newton map is described in Proposition 2.2, \cite{Nayak_Pal2022}, which we restate here. 

\begin{lem}\label{NFP}
Let $R=\frac{p}{q}$, where $p$ and $q$ are polynomials without any common factor and with respective degrees $d$ and $e$. If $\alpha$  is a root of $R$ with multiplicity $k$ and $\beta$ is a pole of $R$ with multiplicity $l$, then we have the following.
	\begin{enumerate}[(a)]
		\item $\alpha$ is an attracting fixed point of $N_R$ with multiplier $\frac{k-1}{k}$.
		\item $\beta$ is a repelling fixed point of $N_R$ with multiplier $\frac{l+1}{l} $.
		\item $\infty$ is a fixed point of $N_R$ if and only if $d\neq e+1$, and in that case, the multiplier of $\infty$ is $ \frac{d-e}{d-e-1}. $ Therefore, $\infty$ is attracting if $d \leq e$ (superattracting if $d=e$) and repelling if $d>e$.
	\end{enumerate}
\end{lem}
\begin{rem}\label{Atlst_attracting}
 Each finite root of $R$ is an attracting fixed point of $N_R$. Even if $R$ has no finite root, i.e., $R(z)=\frac{1}{q(z)}$, where $q$ is a polynomial of degree at least two, then it follows from Lemma \ref{NFP}(3) that $\infty$ is an attracting fixed point. Therefore, every Newton map has at least one attracting fixed point.
\end{rem}
As evident from Lemma~\ref{NFP}, the multiplier of each fixed point of a Newton map is of the form $\frac{r}{s}$, where $r$ and $s$ are integers such that $|r-s|=1$. This leads to a characterization of all Newton maps.
\begin{thm}[Characterization of Newton maps, \cite{Nayak_Pal2022}]\label{Characterization_N_R}
	Let $N$ be a rational map of degree at least two. Then $N=N_R$ for a rational function $R$ if and only if all the fixed points of $N$ are simple (i.e., a simple root of $N(z)-z=0$) and all but one of their multipliers are of the form $ \frac{r}{s} $ for some $r \in \mathbb{N} \bigcup \{0\}, s \in \mathbb{N}$ with $|r-s|=1$. Moreover, each finite fixed point of $N$ with multiplier $\frac{r}{s}$ is either a root (if $r < s$) or a pole (if $r> s$) of $R$ with multiplicity $s$.
\end{thm}
\begin{cor}\label{Characterization_poly}
	If a rational function $N$ has exactly one repelling fixed point and the multipliers of all fixed points are either $0$ or of the form $\frac{r}{s}$ for some $r,~s \in \mathbb{N}$ with $|r-s|=1$, then $N$ is conjugate to the Newton's method applied to a polynomial.
\end{cor}
Suppose $z_0$ is a fixed point of rational function $R$. Then the residue fixed point index of $R$ at the fixed point $z_0$ is defined as 
$$
    \iota (R,z_0)= \frac{1}{2 \pi i} {\oint\limits_{\gamma}\frac{1}{z-R(z)}\,\mathrm{d}z},
$$
where $\gamma$ is a small positively oriented closed curve around $z_0$ that does not surround any other fixed point of $R$. If $z_0$ is a simple fixed point with multiplier $\lambda$, then $\iota (R,z_0)=\frac{1}{1-\lambda}$. The sum of the residue fixed point indices of all the fixed points of a rational map is always the same.

\begin{thm}[Theorem 12.4, \cite{Milnor_book}]
    For a rational function $R$ with degree at least two, the sum of residue fixed point indices of all its fixed points in $\widehat{\mathbb{C}}$ is $1$.
    \label{rational-fixedpoint-theorem}
\end{thm}

We mention a remark before stating a useful consequence of Theorem~\ref{rational-fixedpoint-theorem}.
\begin{rem}
	If $\alpha$ is a pole of a rational function $R$ with multiplicity $l$, then it is a repelling fixed point of $N_R$ with residue index $-l$. Thus, the sum of residue indices of all finite repelling fixed points of $N_R$ is equal to the negative of the degree of the denominator of $R$.
\end{rem}
\begin{cor}\label{ext_one_attr}
    If a Newton map of degree at least two has exactly one attracting fixed point, then that fixed point cannot be superattracting.
\end{cor}
\begin{proof}
	Each fixed point of a Newton map is either attracting or repelling. The multiplier of a repelling fixed point is $\frac{l +1}{l}$ for some integer $l  \geq 1$, and  its residue index is $-l$.
	Letting $\lambda$ to be the multiplier of the attracting fixed point of the Newton map, it follows from Theorem~\ref{rational-fixedpoint-theorem} that $\frac{1}{1-\lambda} >1$. In other words, $\lambda>0$. Thus, the attracting fixed point is not superattracting. 
\end{proof}
\section{Some conformally equivalent Newton maps}\label{Sec3_conf_equiv}
Theorem~\ref{two-attracting-fixedponts} determines all Newton maps (up to conjugacy) with exactly two attracting fixed points, one of which is exceptional. Besides proving this theorem, we discuss some properties of Newton maps with at least one exceptional point. We also determine all the Newton maps with exactly one repelling fixed point, which can be of independent interest.
\begin{prop} 
	A Newton map with degree at least two has exactly one repelling fixed point if and only if it is conjugate to $N_{p}$ for some monic polynomial $p$ with at least two distinct roots. 
	\label{conj_to_poly}
\end{prop}
\begin{proof}
Let $z_0$ be the repelling fixed point of a Newton map $ N_{R}$ for some rational function $R$. If $z_0 = \infty $, then $R$ cannot have any finite pole, and therefore $R$ is a polynomial. In view of the Scaling property, $N_R$ is conjugate to $N_p$ where $p$ is a monic polynomial. Since the degree of $N_R$ is at least two, $p$ has at least two distinct roots (see Remark~\ref{NM-linear}(1)). If $z_0 $ is finite, then considering  $\psi(z) =\frac{1}{z-z_0}$,  $\psi \circ N_R \circ \psi^{-1}$ is a Newton map by Theorem \ref{Characterization_N_R}. This map has exactly one repelling fixed point, and that is $\infty$. Now, we are done as in the previous case (i.e., $z_0 = \infty$).

\par 
Conversely, if $p$ is a monic polynomial with at least two distinct roots, then the degree of $N_p$ is at least two, and it has exactly one repelling fixed point.
\end{proof}
Every exceptional point is either a fixed point or a $2$-periodic point of a rational function. In both cases, it is superattracting. We show that a Newton map can have exactly two exceptional points only when it is quadratic. 
\begin{prop}[Two exceptional points]\label{two_exp}
	If a Newton map with degree at least two has two exceptional points, then it is  conjugate to $z^2$.
	\label{quadratic-newton}
\end{prop}
\begin{proof}
	If $N$ is a Newton map with degree at least two having exactly two exceptional points, then it is conjugate either to $\frac{1}{z^d}$ or to $z^d$, where $d$ is the degree of $N$ (see Theorem 4.1.2, \cite{Beardon_book}). Since all the fixed points of  $\frac{1}{z^d}$, are repelling (more precisely, each has its multiplier equal to $-d$), it follows from Remark \ref{Atlst_attracting} that $N$ cannot be conjugate to $\frac{1}{z^d}$. Therefore, $N$ is conjugate to $z^d$. For $d \geq 2$, since the multiplier of every non-zero fixed point of $z^d$ is $d$, $N$ has a fixed point with multiplier $d$. It follows from  Theorem~\ref{Characterization_N_R} that if the multiplier of a fixed point of a Newton map is a non-zero integer, then it must be $2$. Thus,  $d= 2$. 
\end{proof}
A Newton map with degree at least three cannot be conjugate to $z^2$ and that gives rise to the following.
\begin{cor}
A Newton map with degree at least three has at most one exceptional point.
\end{cor}
There is a remark. 
\begin{rem}
	If a Newton map $N$ with degree exactly two has two exceptional points, then it is conjugate to $N_{R_i}, i \in \{1,2,3\}$, where $R_1 (z)=z(z-1)$, $R_2 (z)= \frac{z}{z-1}$, and  $R_3 (z)=\frac{z^2-1}{z}.$ To see it, note that $N$ has three fixed points, say $a,b,c$ with respective multipliers $0,0$ and $2$. Let $N=N_R$ for some rational function $R$. 
	\par If $c=\infty$, then $R$ is a polynomial. By the Scaling property, $N_R$ is conjugate to $N_{R_{1}}$. If $c \neq \infty$, then $R$ has a finite pole. There are two possibilities: one of $a, b$ is $\infty$, or both are finite. In the first case, assuming $a= \infty$, we see that $N$ is conjugate to $N_{R_2}$. Similarly, $N$ is conjugate to $N_{R_3}$ in the other case.
\end{rem}

Here is an observation on Newton maps arising out of polynomials.
\begin{prop}
	Let $d \geq 2$ and $p$ be a polynomial with $d$ distinct roots. Further, let there be a single simple root and all other roots have the same multiplicity, say $m \geq 1$. If the Newton map $N_p$ has an exceptional point, then it is conjugate to $N_{z(-1+z^{d-1})^m}$.
	\label{polynewtonmap-conj-poly}
\end{prop}
\begin{proof}
The exceptional point of $N_p$ is assumed to be $0$ without loss of generality in view of the Scaling property. Then $p$ can be represented as $p(z)=z(q(z))^m$, where $q$ is a generic polynomial with $q(0)\neq 0$, and
	$$
	N_p(z)= z-\frac{zq(z)}{q(z)+mzq'(z)}=\frac{mz^2q'(z)}{q(z)+mzq'(z)}.$$
	As $0$ is an exceptional point of $N_p$ and the degree of $N_p$ is $d$, we have $q'(z)=\lambda z^{d-2}$ for some $\lambda\neq 0$. Therefore,  $q(z)=\frac{\lambda}{d-1}z^{d-1}+c$, for some non-zero constant $c$. Now, using the Scaling property, we can take $c=\frac{\lambda}{d-1}$ and hence $N_p$ is conjugate to $N_{z(-1 +z^{d-1})^m}$.
\end{proof}
We now present the proof of Theorem~\ref{two-attracting-fixedponts}.
  
\begin{proof}[Proof of Theorem~\ref{two-attracting-fixedponts}]
Let $z_1, z_2$ be the two attracting fixed points of a Newton map $N$ such that $z_2$ is exceptional. Then considering $\phi(z)=\frac{z-z_1}{z-z_2}$, it is seen that $\infty$ and $0$ are the only attracting fixed points of $\phi \circ N \circ \phi^{-1}$ and $\infty$ is exceptional. The map  $\phi \circ N \circ \phi^{-1}$ is also a Newton map by Theorem~\ref{Characterization_N_R}.  In particular,  $\phi \circ N \circ \phi^{-1}$ is a polynomial. If  $R$ is a rational function such that $\phi \circ N \circ \phi^{-1}=N_R$, then $0$ is the only finite root of $R$. Since $N_R$ has at least one repelling fixed point, there is a pole of $R$. Thus, $R(z)=\frac{c z^d}{p(z)},$ where  $d \geq 1, c  \neq 0$ and $p$ is a polynomial with $p(0) \neq 0$. Since $\infty$ is a superattracting fixed point of $N_R$, the degree of $p$ is $d$, by Lemma~\ref{NFP}(2). Further, in view of the Scaling property, we can take $p$ to be monic and $c=1.$ Thus, $R(z)=\frac{z^d}{p(z)}$ for a monic $p$ with $p(0) \neq 0$ and $\deg(p)=d$. Therefore,
$$N_R(z)=z-\frac{z p(z)}{d p(z)-z p'(z)}.$$
Letting $p(z)=\prod_{i=1}^{k} (z-\alpha_i)^{m_i} $ where each $\alpha_i$ is a  root with multiplicity $m_i \geq 1$, we observe that,
$$dp(z)-zp'(z)= \prod_{i=1}^{k} (z-\alpha_i)^{m_i -1} \left(d \prod_{i=1}^{k} (z-\alpha_i)  -z \sum_{i=1} ^{k} m_i \left(\prod_{j \neq i} (z-\alpha_j)\right) \right) .$$ 

Take 
\begin{equation} g(z) =d \prod_{i=1}^{k} (z-\alpha_i)  -z \sum_{i=1} ^{k} m_i \left(\prod_{j \neq i} (z-\alpha_j)\right),
\label{g}\end{equation}  and note that $g(z)$ and $zp(z)$ have no common factor.  Therefore,
\begin{equation}
   N_{\frac{z^d}{p(z)}} ~\mbox{is a polynomial only when } g(z)   ~\mbox{is a non-zero constant.}
   \label{g-constant} 
\end{equation} 

\begin{enumerate}
\item Let $p$ be generic. If $p$ is linear, then it follows from the Scaling property that $N_R(z)$ is conjugate to $N_{\frac{z}{z-1}}$. 
\par Let $p$ be non-linear. As $p$ is generic and $N_{ z^d/p(z)}$ is a polynomial, we have $zp'(z) -dp(z)=\alpha $ for some non-zero $\alpha$. Letting  $y=p(z)$,  we have the first-order linear differential equation $y'-\frac{d}{z}y=\frac{\alpha}{z}$ .  The solution is
    $$
    \begin{aligned}
    	y \cdot \frac{1}{z^d}
    	&=\int \frac{\alpha}{z} \cdot \frac{1}{z^d} d z+\beta=-\frac{\alpha}{d z^d } +\beta,
    \end{aligned}
    $$
  for an arbitrary constant $\beta$.
    Therefore, $ p(z)=\beta z^d-\frac{\alpha}{d}$ and $R(z)=\frac{z^d}{ \beta z^d -\frac{\alpha}{d}}$. Consider $c$ such that $c^d = \frac{\alpha}{d \beta}$ and use the
  Scaling property to see that $N_R$ is conjugate to $N_{\beta R ( cz)}$. We are done since $\beta R(cz)= \frac{z^d}{ z^d -1}$ and the resulting Newton map is $\frac{z^{d+1} +(d-1)z}{d}$.
 
 \item Let $n$ denote the number of distinct roots of $p$. Using the Scaling property, we assume without loss of generality that $1$ is a multiple root of $p$ whenever $p$ is not generic. Along with this, what is going to be repeatedly used in all the following cases is that $g(z) $ is a non-zero constant (see Equation~(\ref{g})).
{\begin{enumerate}
		\item Let $\deg(p)=3$. Then there are three cases depending on the values of $n$. If $n=3$, then $p$ is generic, and from the first part of this theorem, it follows that $p(z)=z^3-1$, and hence $N_R(z)=\frac{1}{3}z(z^3+2)$. If $n=2$, then  $p$ has a root with multiplicity $2$, and therefore  $p(z)=(z-1)^2(z-a)$, where $a\neq 0,1$. In this case, $g(z) =-(a+2)z+3a$, and therefore  $a=-2$. Thus,   $N_R(z)= \frac{z}{6}(z^2+z+4)$. If $n=1$, then  $p(z)=(z-1)^3$, and we get  $N_R(z)=\frac{1}{3}z(z+2)$. 
		\item Let $\deg(p)=4$. All possible cases of $p$ and the resulting Newton maps are given in Table~\ref{d=4}.
		
 \begin{table}
\centering			
\begin{tabular}{|>{\centering\arraybackslash}m{0.5cm}|>{\centering\arraybackslash}m{3.5cm}|>{\centering\arraybackslash}m{4.5cm}|>{\centering\arraybackslash}m{2.5cm}|>{\centering\arraybackslash}m{2.5cm}| }
\hline
\textbf{ $n$} & \textbf{$p(z)$}& $g(z)$&  $a$,   $b$& $N_{\frac{z^4}{p(z)}}(z)$  \\ 
\hline
\hline
4 			& $z^4 -1$			&\centering -		 &\centering - &$\frac{z(z^4 +3)}{4}$\\ 
\hline
3 			&  $(z-1)^2(z^2+az+b)$&$(a-2)z^2+(-3a+2b)z-4b$& $a=2, b=3$  &$\frac{z(z^3+z^2+z+9)}{12}$  \\ 
\hline
2&(i) $(z-1)^2(z-a)^2$& $2(2a-(a+1)z)$&$a=-1$&$\frac{z(z^2 +3)}{4}$ \\ 
\hline
2 &(ii) $(z-1)^3(z-a) $& $-(a+3)z+4a$&$a=-3$&$\frac{z(z^2+2z+9)}{12}$\\ 
\hline
1 &  $(z-1)^4$&4&-  &$\frac{z(z+3)}{4}$\\ 
\hline		
\end{tabular}
\caption{Newton  maps  $N_{\frac{z^4}{p(z)}}$  with an exceptional point}
\label{d=4}
\end{table}
\item For $\deg(p)=5$, all the possible forms of $p$ and resulting Newton maps are given in Table~\ref{d=5}.
	\begin{table}
		\centering			
	\begin{tabular}{|c |c|c|c|c| }
		\hline
\textbf{ $n$} & \textbf{$p(z)$}& $g(z)$&  $a,b, c$& $N_{\frac{z^5}{p(z)}}(z)$  \\ 
		\hline
		\hline
5 & $z^5 -1$			&-		 &-		  &$\frac{z(z^5 +4)}{5}$\\ 
		\hline
4& $(z-1)^2 (z^3 +$	&$(a-2)z^3+(-3a+2b)z^2 $ &$a=2$ 	& \\ 
 & $ a z^2+bz+c)$	&$ +(-4b+3c)z-5c$		 & 	  $b=3$ &$\frac{z(z^4 +z^3 +z^2 +z+16)}{20}$\\ 
  &  				& 						 & 	  $c=4$  &			\\ 
\hline
3& (i)$(z-1)^2(z-a)^2$&$-(2a+b+2)z^2$&$a=\frac{-2+ i\sqrt{5}}{3}$& $\frac{z}{10(\sqrt{5}-7i)} (-9i z^3 +$\\
	&  $(z-b)$ 		&$ (4a+3b+3ab)z-5ab$&$b= \frac{-2- i 2 \sqrt{5}}{3}$& $(3\sqrt{5}-3i)z^2+ $	\\ 
	&   			& 					&					 	&   $(-\sqrt{5}-2i)z+  $  		\\ 
	&   			& 					&	 					&	$ 8 \sqrt{5}-56i)$	   		\\ 
\hline
3&(ii)$(z-1)^2(z-a)^2$&$-(2a+b+2)z^2$&$a=\frac{-2- i\sqrt{5}}{3}$&          $\frac{z}{10(\sqrt{5}+7i)} (9i z^3 +$\\
&  $(z-b)$ 		&$ (4a+3b+3ab)z-5ab$	&		$b= \frac{-2+ i 2 \sqrt{5}}{3}$	&$(3\sqrt{5}+3i)z^2+ $	\\ 
&   	    	&					 	&									 	& 	$(-\sqrt{5}+2i)z+  $\\ 
&   	    	&					 	&									 	& 	$8\sqrt{5}+56i)   $\\ 
\hline
3&(iii)$(z-1)^3$& $(a-3)z^2 -2(2a -b)z$&$a= 3$& $\frac{z }{30}(z^3 +2z^2$ \\ 
 & $(z^2 +az+b)$& $-5b$				&$b=6$ & $  +3z+24)$ \\ 
\hline
2&(i)$(z-1)^3(z-a)^2$& $(2a+3)z-5a$&$a= -\frac{3}{2}$& $\frac{z(2z^2 +z +12)}{15}$ \\ 
 & 			&			&				 &						   \\ 
\hline
2 & (ii)$(z-1)^4(z-a)$& $(a+4)z-5a$   &$a=-4$ & $\frac{z(z^2 +3z +16)}{20}$ \\ 
 &  & &  &  \\ 
\hline
1 &  $(z-1)^5$&5&-& $\frac{z(z+4)}{5}$\\ 
\hline		
	
\end{tabular}
\caption{Newton  maps  $N_{\frac{z^5}{p(z)}}$  with an exceptional point }
\label{d=5}	
\end{table}
 \end{enumerate}}	
 \end{enumerate} 
Hence the proof is complete.
\end{proof}

\section{Totally disconnected Julia sets}\label{Sec4_totally_discon}
This section discusses a class of Newton maps with totally disconnected Julia sets. We start with a basic observation on Newton maps with totally disconnected Julia sets.
\begin{prop}
If the Julia set of a Newton map is totally disconnected, then the Newton map is not conjugate to any polynomial.
\end{prop}
\begin{proof}
Recall that the fixed points of every Newton map are either attracting or repelling. Since the Julia set is totally disconnected, the Newton map has exactly one attracting fixed point. However, this attracting fixed point is not superattracting (see Corollary \ref{ext_one_attr}).
As $\infty$ is always a superattracting fixed point of every non-linear polynomial and conjugacy preserves the multiplier of fixed points, we conclude that the Newton map is not conjugate to any polynomial.
\end{proof}
 
\par  
 To prove Theorem~\ref{diconnected_gen}, we need the following lemmas. For an attracting fixed point $z_0$ of a rational function $R$, the attracting basin of $z_0$ is the set $\{z: R^n (z) \to z_0~\mbox{as} ~n \to \infty\}$. This set is always open, and its connected component containing $z_0$ is known as the immediate basin of $z_0$, which is denoted by $\mathcal{A}_{z_0}$.

\begin{lem}\label{totally_disconnected} (Theorem 7.5.1, Theorem 9.8.1, \cite{Beardon_book})
    If a rational function $R$ with degree at least two has an invariant immediate basin $U$, then $U$ contains at least one critical point of $R$. Moreover, if $U$ contains all the critical points of $R$, then the Julia set of $R$ is totally disconnected.
\end{lem}
\begin{lem}
	If $\psi$ is a homeomorphism such that $\psi^{-1} \circ R \circ \psi =R$ for a rational function $R$, then $\psi(\mathcal{J}(R)) =\mathcal{J}(R)$.
	\label{symmetry} 
	\end{lem}
	The proof of the above lemma is straightforward (see Theorem 3.1.4, \cite{Beardon_book}). 

\begin{proof}[Proof of Theorem \ref{diconnected_gen}]
	First, we apply the Scaling property to arrive at different Newton maps up to conjugacy.
	\par  If $p$ has two distinct roots, then $p(z)=c(z-a)^m(z-b)^n$, where $a,b,c \in \mathbb{C}, a\neq b, c \neq 0$ and $m,n \geq 1$. For $T(z)=(b-a)z+a$ and $\lambda =c (b-a)^{m+n}$, we have that $N_{1/p}$ is conjugate to $N_{ \lambda/p(T(z))}=N_{1/z^m (z-1)^n}$.
	If $p$ is uncritical, i.e., $p(z)=c(z-a)^n +d$, for some $c,a, d\in \mathbb{C}, c \neq 0$, and $n \geq 2$, then $d \neq 0$ by assumption. Considering $T(z)=\alpha z+a$, where $\alpha^n=\frac{d}{c}$, and $\lambda =d$, we see that $N_{1/p}$ is conjugate to $N_{ \lambda/p(T(z))}=N_{1/z^n +1}$. In the third case, taking $T(z)=\alpha z, \alpha^n =a$ and $\lambda =\alpha a$, we see that $N_{1/p}$ is conjugate to  $N_{1/z(z^n +1)}$. Denoting the Newton maps in the first, second, and third cases as $N_0, N_1$, and $N_2$ respectively, we have
    \begin{align}
       & N_0(z)=\frac{z((m+n+1)z-(m+1))}{(m+n)z-m},\\
       & N_1(z)=\frac{(n+1)z^n+1}{nz^{n-1}}, ~~\mbox{and} \\
       & N_2(z)=\frac{z((n+2)z^n+2)}{(n+1)z^n+1}.
    \end{align}
Since all the coefficients of $N_i$, $i=0,1,2$, are real, we have $\overline{N_i(z)}=N_i(\bar{z})$, i.e., $\mathcal{J}(N_i)$ is symmetric about the real axis (see Lemma~\ref{symmetry}). Further, the point at $\infty$ is an attracting fixed point for each $N_i$. Let $\mathcal{A}_\infty$ denote the immediate basin of attraction of $\infty$ in each case. Moreover, for $j=1,2$, we have $N_j(\lambda z)=\lambda N_j(z)$ whenever $\lambda^n=1$. Therefore, 
$\{z\mapsto \lambda z: \lambda^n=1 \}\subseteq \Sigma N_j, j=1,2.$  Observe that,
\begin{equation}\label{sym_New}
\mathcal{A}_{\infty}~\mbox{is preserved under}~ z \mapsto \lambda z ~\mbox{where}~ \lambda^n =1, ~\mbox{as well as under}~ z \mapsto \overline{z}.\end{equation}

	To prove that $\mathcal{J}(N_i)$ is totally disconnected, it is enough to show that all the critical points of $N_i$ are in $\mathcal{A}_\infty$ (by Lemma \ref{totally_disconnected}).
	\begin{enumerate}
    \item As $\deg(N_0)=2$ and $N_0'(z)=\frac{(m+n+1)(m+n)z^2-2m(m+n+1)z+m(m+1)}{((m+n)z-m)^2}$, we have $N_0$ has exactly two critical points, namely the roots of $$(m+n+1)(m+n)z^2-2m(m+n+1)z+m(m+1)=0.$$ As $(-2m(m+n+1))^2-4(m+n+1)(m+n)m(m+1)=-4mn(m+n+1)<0,$ these two critical points are complex conjugates to each other. As $\mathcal{A}_\infty$ must contain a critical point by Lemma~\ref{totally_disconnected}, it contains both of them by Equation~(\ref{sym_New}). Hence, $\mathcal{A}_\infty$ contains all the critical points of $N_0$.
	\item The origin is the only pole of  $N_1$, and it is in the attracting basin of $\infty$. Note that $N_1'(z)=\frac{(n+1)z^n-(n-1)}{nz^n}$. If $n=2$, then the critical points of $N_1$ are $\pm \frac{1}{\sqrt{3}}$. One of these must be in $\mathcal{A}_\infty$. Therefore, the other is also in $\mathcal{A}_\infty$ by Equation~(\ref{sym_New}). 
	\par Let $n\geq 3$. Then the pole (at the origin) is a critical point of $N_1$ with multiplicity $n-2$. The origin is in the basin of $\infty$. Other critical points of $N_1$ are solutions of $z^n=\frac{n-1}{n+1}$. Among them, one is  positive. Let it be  $c$ (see Figure (\ref{cr_graph_N1})). Note that $N_1(x) >x$ and $N_1 '(x)>0$ for all $x >c$. This gives that  
	$\lim_{k\to \infty}N^k_1(x)=\infty$ for all $x>c$. Since $N_1 (c)>c$, we have $\lim_{k\to \infty}N^k_1(c)=\infty$. In other words, $c \in \mathcal{A}_{\infty}$.
	It follows from Equation~(\ref{sym_New}) that all other $(n-1)$ many critical points are in $\mathcal{A}_\infty$.
	
    \begin{figure}[h!]
	\begin{subfigure}{.5\textwidth}
		\centering
		\includegraphics[width=1\linewidth]{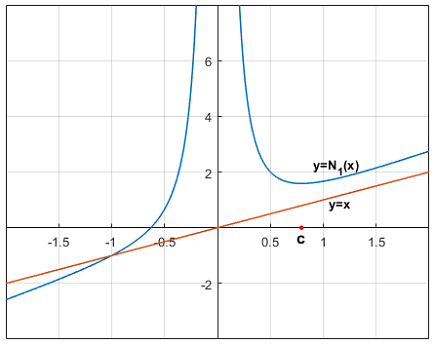}
		\caption{$n=3: \frac{4z^3 +1}{3z^2}$}
	\end{subfigure}%
	\begin{subfigure}{.5\textwidth}
		\centering
        \includegraphics[width=1\linewidth]{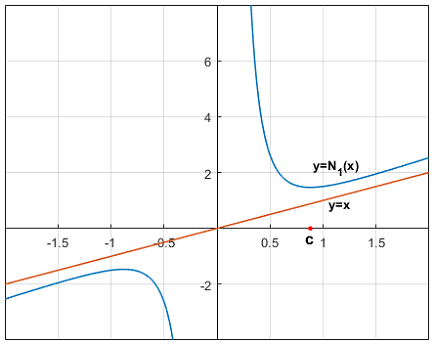}
		\caption{$n=4:\frac{5z^4 +1}{4z^3} $}
	\end{subfigure}
	\caption{The graph of $N_1$: There is a negative fixed point for odd $n$ and this is not the case for even $n$.}
    \label{cr_graph_N1}
\end{figure}
	\item Note that 
	\begin{equation}\label{crit_N2}
	N_2'(z)=\frac{(n+1)(n+2)z^{2n}-(n+1)(n-4)z^n+2}{((n+1)z^n+1)^2}.
	\end{equation}
	Thus the critical points of $N_2$ are the solutions of $z^n=c_1$ or $z^n=c_2$, where $c_1=\frac{(n+1)(n-4)+n\sqrt{(n+1)(n-7)}}{2(n+1)(n+2)}$ and $~c_2=\frac{(n+1)(n-4)-n\sqrt{(n+1)(n-7)}}{2(n+1)(n+2)}.$
	\par 
	If $n<7$, then $c_2=\bar{c}_1$ (see Figure (\ref{cr_graph_N2}(a))). The immediate basin $\mathcal{A}_\infty$ contains at least one critical point. Without loss of generality, let it be an $n$-th root of $c_1$. Then all other $n$-th roots of $c_1$ are also in $\mathcal{A}_\infty$ (by Equation~(\ref{sym_New})).
Since each $n$-th root of $c_2$ is conjugate to an $n$-th root of $c_1$, we have that $\mathcal{A}_\infty$ contains all the critical points of $N_2$.
	\par 
	If $n=7$, then $c_1=c_2=\frac{1}{6}$ (see Figure (\ref{cr_graph_N2}(b))) and the critical points are precisely the $n$-th roots of $\frac{1}{6}$. Each is with multiplicity two. As $\mathcal{A}_\infty$ contains at least one critical point, it contains all other critical points by Equation~(\ref{sym_New}).
    \begin{figure}[h!]
	\begin{subfigure}{.5\textwidth}
		\centering
		\includegraphics[width=1\linewidth]{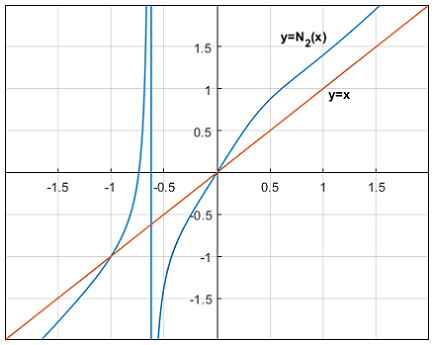}
		\caption{$n=3$: No positive critical point}
	\end{subfigure}%
	\begin{subfigure}{.5\textwidth}
		\centering
        \includegraphics[width=1\linewidth]{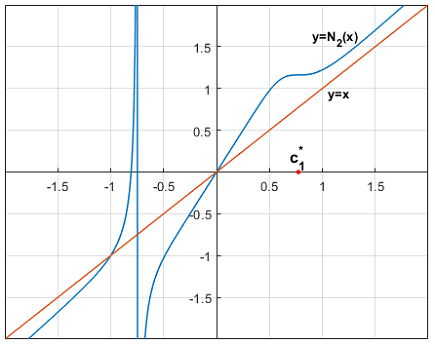}
		\caption{$n=7$: Exactly one positive critical point}
	\end{subfigure}\\[1ex]
	\centering
	\begin{subfigure}{0.5\textwidth}
		\centering
		\includegraphics[width=1\linewidth]{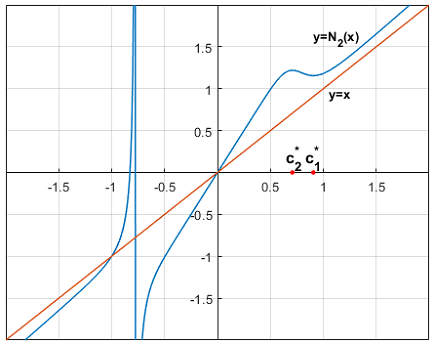}
		\caption{$n=9$: Two distinct positive critical points }
	\end{subfigure}
	\caption{The graph of $N_2$}
    \label{cr_graph_N2}
\end{figure}
	\par 
For $n>7$, Equation~(\ref{crit_N2}) can be written as \begin{equation}\label{Re_crit_N2}
        N_2'(z)=\frac{(n+1)(n+2)(z^n-c_2)(z^n-c_1)}{((n+1)z^n+1)^2}.
    \end{equation}
    Note that $0<c_2<c_1$. Let $c_1^*$ and $c_2^*$ be the positive real solutions of $z^n=c_1$ and $z^n=c_2$, respectively. Then, $0<c_2^*<c_1^*$ (see Figure (\ref{cr_graph_N2}(c))) and from Equation~(\ref{Re_crit_N2}), we get that $N_2$ is increasing in $(0,c_2^*)\cup (c_1^*,\infty)$ and it is decreasing in $(c_2^*, c_1^*)$. Therefore, $N_2(c_1^*)<N_2(c_2^*)$.
For all $x > c_1 ^*$,  $N_2 (x)>x$ and $N_2 '(x)>0$, and therefore $\lim\limits_{k \to \infty} N_2 ^{k}(x)=\infty $. As $N_2 (c_1 ^*)> c_1 ^*$,  we also have $\lim\limits_{k \to \infty} N_2 ^{k}(c_1 ^*)=\infty $. In other words, $[c_1^*, \infty)$ is contained in $\mathcal{A}_{\infty}$.
Now $N_2([c_2^*,c_1^*])=[N_2(c_1^*),N_2(c_2^*)] \subseteq (c_1^*,\infty)$, implies that $[c_2^*,c_1^*]\subset \mathcal{A}_{\infty}$.  It follows from Equation~(\ref{sym_New}) that  $\mathcal{A}_\infty$ contains all the critical points of $N_2$.
     
	\end{enumerate}
 The proof of the theorem is complete.
\end{proof}
\begin{rem}
	From the proof of the second case above, it is clear that $(0, \infty) \subset \mathcal{A}_\infty$. This is  because the minimum value of $N_1$ in $(0,c)$ is attained at $c$ and therefore, $N_1$ takes $(0,c)$ into $(c, \infty)$. 
\end{rem}
\section{Connected Julia sets}\label{Sec5_Conn_J_set}
This section discusses two classes of Newton maps with connected Julia sets.
\subsection{Newton maps with two attracting fixed points}
The length of a polynomial $q(z)=b_1 z^d +b_2 z^{d-1}\cdots +b_{d-1}z+b_d $, denoted by $L(q)$, is defined as $\sum_{i=1}^{d} |b_i|$. For polynomials with an attracting fixed point at the origin, there is always a disk around the origin contained in the  immediate basin of $0$. We estimate the radius of such a disk in terms of the multiplier of $0$ and the length of the polynomial.

\begin{lem}\label{disk_in_imm}
	Let $p$ be a polynomial of degree $d$ at least two with an attracting fixed point at the origin. Define a positive real number $r$ as
	\begin{equation*}
	r = \left\{\begin{array}{ccc}
	\frac{1- |p'(0)|}{L(p)-|p'(0)|} & \text {if } & L(p) \geq 1; \\
	\\
	\left(\frac{1- |p'(0)|}{L(p)-|p'(0)|}\right)^{\frac{1}{d-1}}& \text {if  } & L(p) <1,
	\end{array}\right.
	\end{equation*} where $(.)^{\frac{1}{d-1}}$ denotes the positive $(d-1)$-th root.  Then the immediate basin of the origin contains the disk $\{z: |z|< r\}$.
\end{lem}
\begin{proof}
Let $p(z)=a_1 z^d +\cdots +a_{d-1}z^2+a_d z$. Then $L(p)=\sum_{i=1}^{d} |a_i|$, and
$$
\left|\frac{p(z)}{z}\right| =  \left|a_1 z^{d-1}+\cdots+a_{d-1}z+a_d\right| \\
\leq \left|a_1\right||z|^{d-1}+\cdots+\left|a_{d-1}\right||z|+|a_d|.
$$
\par If $L(p) \geq 1$, then $r \leq 1$ and $r^{m} \leq r$ for all $m=1,2,3, \dots, (d-1)$. Now, for $0<|z|<r$ we have,
$$\left|\frac{p(z)}{z}\right| < r \left( L(p) - |p'(0)| \right)  +|p'(0)| 
 =1.$$
Thus, $|p(z)|<|z|$ for all $z$ with $ 0< |z|<r$.
	
	\par 
	If $L(p) < 1$, then $r>1$ and $r^{m} \leq  r^{d-1}$ for $m =1,2,3,\dots, (d-1)$. For $1<|z|<r$, we have  
$$	\left|\frac{p(z)}{z}\right|  <r^{d-1}(L(p)-|p'(0)|)+|p'(0)|
 =1.
$$
	Thus, $|p(z)|<|z|$ whenever $1<|z|<r$. Now if $0<|z| \leq 1$, then we have, 
$$ \left|\frac{p(z)}{z}\right|=  \left|a_1 z^{d-1}+\cdots+a_{d-1}z+a_d\right|  \leq L(p) < 1.
$$
It follows from Schwarz Lemma that $p^n (z) \to 0$ as $n \to \infty$ for every
$|z|<r$ irrespective of whether $L(p) \geq 1$ or $L(p) <1$. Here $r$ is as defined in the statement of this lemma and depends on $L(p)$.	Hence, we conclude that $\{z: |z|<r\}$ is contained in the immediate basin of the origin.
\end{proof}

\begin{proof}[Proof of Theorem \ref{connected-JS}]
If a Newton map $N$ has exactly two attracting fixed points, one of which is an exceptional point, then it follows from Theorem~\ref{two-attracting-fixedponts} that $N=N_R$ where $R(z)=\frac{z^d}{p(z)}$ for some $d \geq 1$, some monic polynomial $p$ with degree $d$ and $p(0) \neq 0$. 
\par If all the repelling fixed points of $N_R$ are with multiplier $2$, then all the poles of $R$ are simple, i.e., $p$ is generic. It follows from Theorem~\ref{two-attracting-fixedponts}(1) that $N_R$ is conjugate to  $F(z)=\frac{z}{d}(z^d+d-1)$. The Fatou set of $F$ is invariant under rotations $z \mapsto \lambda z$ with $\lambda^d =1$, by Lemma~\ref{symmetry}. This gives that the immediate basin $\mathcal{A}_0$ of the origin is also invariant under these rotations. The finite critical points of $N_R$ are the solutions of $z^d=-\frac{d-1}{d+1}$, and these are preserved under the aforementioned rotations. Since at least one of these critical points is in $\mathcal{A}_0$, all of them are in $\mathcal{A}_0$, and the Julia set of $F$ is connected. 

\par The multiplier of the non-exceptional attracting fixed point of $N_R$ is $\frac{d-1}{d}$.  If it is at most $\frac{4}{5}$, then $d \leq 5$. In other words, the degree of $p$ is at most $5$. 
There are two possibilities depending on the values of $d$.
\begin{enumerate}
 	\item If $d=2$ or $3$, then $N_R$ is quadratic or cubic. That the Julia set of $N_R$ is connected follows from Theorem 1.1 and Theorem 1.2, \cite{Nayak_Pal2022}.
	\item  If $d=4$ or $5$, then the degree of $N_R$ is at most six. Unless $p$ is generic or the degree of $N_R$ is two or three, $N_R$ is conjugate to $F_i$ for some $ i =1,2,3,4,5$. This can be seen from Table~\ref{d=4} for $F_1$ and from Table~\ref{d=5} for $F_i, i=2,3,4,5$.
	{\begin{enumerate}
		\item  $F_1(z) = \frac{1}{12}z(z^3+z^2+z+9)$,
		\item $ F_2 (z)=\frac{1}{20}z(z^4+z^3+z^2+z+16)$,
		\item $ F_3 (z)=\frac{z}{10(\sqrt{5}- 7i)}(-9 i z^3+(3 \sqrt{5}-3i)z^2 +(-\sqrt{5}- 2i)z+ 8\sqrt{5}- 56i)$,
	 	\item $	F_4 (z)=\frac{z}{10(\sqrt{5}+ 7i)}(9 i z^3+(3 \sqrt{5}+ 3 i)z^2 +(-\sqrt{5}+ 2i)z+8\sqrt{5}+56 i) $, and
		\item $F_5 (z)= \frac{1}{30}z(z^3+2z^2+3z+24)$.
	\end{enumerate} }

Since $L(F_i)=1$ for $i=1,2,5$, the immediate basin of the origin contains $\{z: |z|<1\}$ in each of these cases by Lemma~\ref{disk_in_imm}. As $L(F_3), L(F_4) >1$, it follows from the same lemma that $\{z: |z|< \frac{2(2 \sqrt{6}-3)}{5}\}$ is contained in the immediate basin of the origin for $F_3$ and $F_4$. We refer to these disks as internal disks and denote them as $\mathbb{D}_i$, where for $i=1,2,5$, $\mathbb{D}_i =\{z: |z|<1 \}$ and for $i=3,4$, $\mathbb{D}_{i}=\{z:|z|< \frac{2(2 \sqrt{6}-3)}{5} \}$. The Julia set of $F_i$ is connected if all the finite critical points of $F_i$ are in the basin of the origin, by Theorem 9.5.1, \cite{Beardon_book}. Thus, we need to show that,
 all the finite critical points or their iterated images are in the respective internal disks for each $i$.

\par 
 We first consider $F_1(z)=\frac{1}{12}z(z^3+z^2+z+9)$. Its real fixed points are $0,1$ and other finite fixed points are the solutions of $z^2+2z+3=0$, which are non-real. In particular, $F_1 (x) > x$ for all $x<0$. There are three finite critical points, exactly one of which, say $c_r$ is real. It is seen that $-2 <c_r <-1$. There is a unique non-zero real root of $F_1$, and that is in $(-3, -2)$. Since $F_1$ is of even degree and has positive leading coefficient, it attains its minimum value at $c_r$ on the real axis, and $F_1$ is increasing in $(c_r,0)$ (see Figure~(\ref{F_graph}(a))). Moreover, $F_1([c_r,0])\subsetneq  (c_r,0]$ and $F_1 ^n (c_r) \to 0$ as $n \to \infty$. This gives that $F_{1}^k (c_r) \in \mathbb{D}_1$ for some $k \geq 1$. The images of the other two non-real critical points of $F_1$ are found to be in $\mathbb{D}_1 $ (see Table~\ref{cr.pt_cr.val}).
 \begin{figure}[h!]
 \begin{subfigure}{.5\textwidth}
 \centering
 \includegraphics[width=0.85\linewidth]{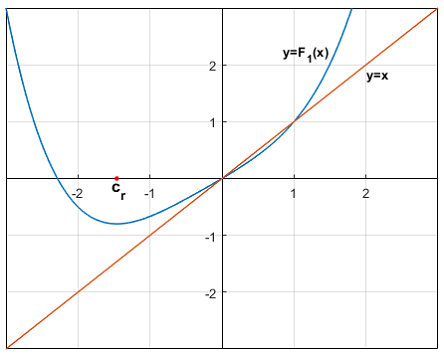}
 \caption{ $F_1 (z) =\frac{1}{12}z(z^3+z^2+z+9) $}
 \end{subfigure}
\hspace*{-1.5cm}
 \begin{subfigure}{.5\textwidth}
\centering
\includegraphics[width=0.85\linewidth]{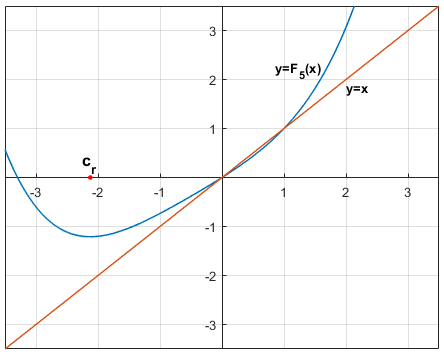}
 \caption{ $F_5(z) = \frac{1}{30}z(z^3+2z^2+3z+24)$}
\end{subfigure}
\caption{Graphs of $F_1$ and $F_2$}
\label{F_graph}
\end{figure}

All the finite critical points of $F_2, F_3$, and $F_4 $ are non-real, and the critical values are in the respective internal disks (see Table~\ref{cr.pt_cr.val}).
\par 
For $F_5(z)=\frac{1}{30}z(z^3+2z^2+3z+24)$, there is exactly one real critical point. Following the same argument as for $F_1$, it is seen that an iterated image of the real critical point is in $\mathbb{D}_5$ (see Figure~(\ref{F_graph}(b))). 
\end{enumerate}	
\end{proof}
\begin{figure}[h!]
	\begin{subfigure}{.5\textwidth}
		\centering
		\includegraphics[width=0.97\linewidth]{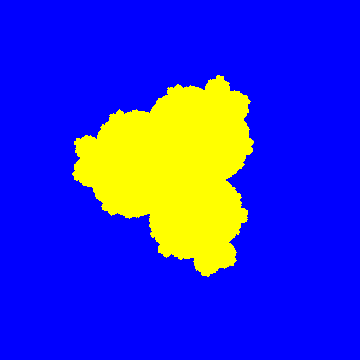}
		\caption{$ F_4$}
	\end{subfigure}%
	\begin{subfigure}{.5\textwidth}
		\centering
		\includegraphics[width=0.97\linewidth]{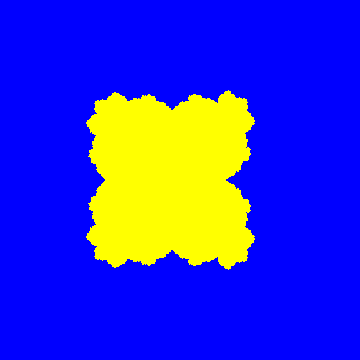}
		\caption{$F_2$}
	\end{subfigure}
	\caption{The  Julia set of $F_4 (z)=\frac{z(9 i z^3+(3 \sqrt{5}+ 3i)z^2+(-\sqrt{5}+ 2i)z+8\sqrt{5}+ 56i)}{10(\sqrt{5}+ 7i)}$ and $F_2 (z)=  \frac{z}{20} (z^4 +z^3+z^2+z+16)$ are shown as the boundary of the yellow and the blue regions in each image.}
	\label{JFset}
\end{figure}
 \newpage
\begin{table}[]
	\centering
	\renewcommand{\arraystretch}{1.5}
	\begin{tabular}{|p{0.6cm}|p{5.5cm}| p{2.0cm}| p{2.0cm}| p{1.5cm}| p{1.4cm}|}
		\hline
		\centering Sl. No. &\centering Newton maps & \centering Critical point ($c$) & \centering Critical value ($c^*$) & \centering $|c^*|$ &  ~~~$r$\\
		\hline
		\hline
		\multirow{3}{10em}{1} & \multirow{3}{10em}{$F_1 (z)=\frac{1}{12}z(z^3+z^2+z+9)$} &$0.355697-1.18874 i$ & $0.115994- 0.678307 i$ & $0.688153$ & \multirow{3}{2em}{\centering$1$}  \\
		\cline{3-5}
		&  & $ 0.355697+ 1.18874 i $ & $ 0.115994+ 0.678307 i $ & 0.688153 &    \\
		\hline
		\multirow{7}{2em}{2} & \multirow{7}{10em}{$F_2 (z)=\frac{1}{20}z(z^4+z^3+z^2+z+16)$} & $ 0.692438- 1.01941 i $ & $ 0.373036- 0.68711 i $ & $ 0.781841 $ & \multirow{7}{2em}{\centering$1$} \\
		\cline{3-5}
		&  & $ 0.692438+ 1.01941 i $ & $ 0.373036- 0.68711 i $ & $ 0.781841 $ &   \\
		\cline{3-5}
		&  & $ -1.09244- 0.955874 i $ & $ -0.69979- 0.5884 i $ & $ 0.914285 $ &   \\
		\cline{3-5}
		&  & $-1.09244+ 0.955874 i$ &$-0.69979+ 0.5884i$ & $0.914285$ &   \\
		\hline
		\multirow{5}{2em}{3} & \multirow{5}{15em}{$F_3 (z)=\frac{z}{10(\sqrt{5}- 7i)}(-9 i z^3+(3 \sqrt{5}-3i)z^2$ $+(-\sqrt{5}- 2i)z+8\sqrt{5}- 56i)$ } & $ 0.426365+ 0.953382 i $ & $ 0.253282+ 0.566356 i $ & $ 0.620412 $ &   \\
		\cline{3-5} 
		&  & $ 0.51694- 1.13864 i $ & $0.237926- 0.699289 i$ & $ 0.73866 $ & \multirow{2}{2em}{\centering $\frac{2(2\sqrt{6}-3)}{5} $}  \\
		\cline{3-5}
		&  & $ -1.19327- 0.373795 i $ & $ -0.67984- 0.28885 i $ & $ 0.73866 $ & $\approx 0.75$  \\
		\hline
		\multirow{5}{2em}{4}  & \multirow{5}{15em}{$F_4 (z)=\frac{z}{10(\sqrt{5}+ 7i)}(9 i z^3+(3 \sqrt{5}+3i)z^2$ $+(-\sqrt{5}+ 2i)z+8\sqrt{5}+ 56i)$}& $ 0.426365- 0.953382 i $ & $ 0.253282- 0.566356 i $ & $ 0.620412 $ &   \\
		\cline{3-5}
		&  & $ 0.51694+ 1.13864 i $ & $0.237926+ 0.699289 i$ & 0.73866 &  \multirow{2}{2em}{\centering $\frac{2(2\sqrt{6}-3)}{5}$}  \\
		\cline{3-5}
		&  & $ -1.19327+ 0.373795 i $ &$ -0.67984+0.28885 i $ & $ 0.73866 $ & $\approx 0.75$   \\
		\hline
		\multirow{3}{10em}{5} & \multirow{3}{10em}{ $F_5 (z)=\frac{1}{30}z(z^3+2z^2+3z+24)$}  & $ 0.311937+ 1.65158 i $ & $ 0.01360+ 0.975419 i $ & $ 0.97551 $ & \multirow{3}{2em}{\centering $1$}  \\
		\cline{3-5}
		&  & $ 0.311937- 1.65158 i $ &$ 0.01360+ 0.97542 i $ & $ 0.97551 $ & \\
		\hline
	\end{tabular}
	\caption{Critical points and radii of internal disks: A non-real critical point of $F_i$ and the corresponding  critical value are denoted by $c$ and $c^*$, respectively. The real number $r$ is the radius of the internal disk $\mathbb{D}_r$. Table~\ref{cr.pt_cr.val} shows a comparison between $|c^*|$ and $r$, which gives that each such critical value is in the immediate basin of the origin.}
	\label{cr.pt_cr.val}
\end{table}
By a Julia point, we mean a point contained in the Julia set. For a rational function $R$, the postcritical set is the union of forward orbits of all the critical points of $R$. If the closure of the postcritical set contains only finitely many Julia points, then $R$ is said to be a geometrically finite map. It is known that if $R$ is geometrically finite with a connected Julia set, then $\mathcal{J}(R)$ is locally connected (Theorem A, \cite{Lei1996}). There is a remark on Theorem~\ref{connected-JS}.
\begin{rem}
	\begin{enumerate}
		\item It follows from the proof of Theorem~\ref{connected-JS} that all the Newton maps are geometrically finite with connected Julia sets. Hence, their Julia sets are locally connected (see Figure \ref{JFset} for example).
\item If $p$ is generic, then  $\mathcal{A}_0$ is simply connected and contains $d$ many critical points counting multiplicity. Therefore, it is completely invariant under the concerned Newton map by the Riemann-Hurwitz formula. Thus, the Fatou set consists of two completely invariant domains $\mathcal{A}_0$ and $\mathcal{A}_\infty$, and  $\mathcal{J}(F)$ is a Jordan curve. 
\end{enumerate}
 \end{rem}
\subsection{Newton's method applied to McMullen maps}
The McMullen map $f_\lambda(z)$ is defined as
$$f_\lambda(z)=z^m-\frac{\lambda}{z^n}=\frac{z^{m+n}-\lambda}{z^n}$$
where $\lambda \in \mathbb{C} \setminus \{0\}$ and $m,n \geq 1$. 

\begin{lem}\label{Conj_N_Mc}
	Let $f_\lambda(z)=z^m-\frac{\lambda}{z^n}$ for  $\lambda \neq 0$ and $m,n \geq 1$. Then the Newton map $N_{f_\lambda}$ is conjugate to $N_{f_1}$.
\end{lem}
\begin{proof}
	Consider the affine map $T(z)=\lambda^{\frac{1}{m+n}} z$, and $c=\lambda^{-\frac{m}{m+n}}$. Then we have
	$$	 cf_\lambda\left(T(z)\right)=c\frac{\lambda\left(z^{m+n}-1\right)}{\lambda^{\frac{n}{m+n}} z^n} = \frac{z^{m+n}-1}{z^n}=f_1(z).
	$$
	By the Scaling property, we have $N_{f_1}=T^{-1}\circ N_{f_\lambda}\circ T$. 
\end{proof}
Now onwards, let $f_1 (z)=f(z)$ for the sake of a simpler notation.
To study the dynamics of the Newton's method applied to a McMullen map $f_\lambda$, it is enough to consider   $f(z)=\frac{z^{m+n}-1}{z^n}$.
 Note that $f'(z)= \frac{m z^{m+n}+n}{z^{n+1}}$ and,
\begin{eqnarray}
N_f(z)&&= z-\frac{z\left(z^{m+n}-1\right)}{m z^{m+n}+n} \label{N1}\\
&& =\frac{z\left((m-1) z^{m+n}+(n+1)\right)}{m z^{m+n}+n}, \label{N2} 
\end{eqnarray}
and
\begin{eqnarray}
N_f^{\prime}(z)    =\frac{( z^{m+n}-1)(m(m-1) z^{m+n}-n(n+1))}{(m z^{m+n}+n)^2} \label{N'}.
\end{eqnarray} 
If $m=1$ then $f(z)=\frac{z^{n+1}-1}{z^n}$ and $S(z)=\frac{1}{N_f(\frac{1}{z})}=\frac{nz^{n+1}+1}{(n+1)z^n}$, which is nothing but  $N_{z^{n+1}-1}$. Therefore, $\mathcal{J}(N_f)$ is connected by Theorem~\ref{shishikura} (see Figure~(\ref{m1n2}(a))).
\begin{figure}[h!]
	\begin{subfigure}{.5\textwidth}
		\centering
		\includegraphics[width=0.98\linewidth]{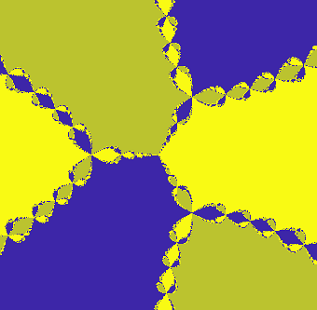}
		\caption{$N_f$, where $f(z)=z-\frac{1}{z^2}$}
	\end{subfigure}
	\begin{subfigure}{.5\textwidth}
		\centering
		\includegraphics[width=0.98\linewidth]{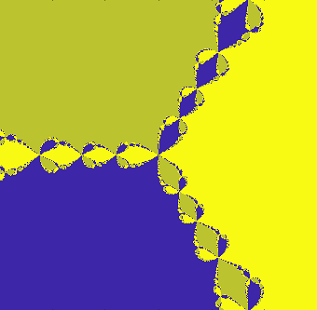}
		\caption{$N_p$, where $p(z)=z^3-1$}
	\end{subfigure}
	\caption{The Fatou set of the  Newton's method of McMullen maps for $m=1$}
	\label{m1n2}
\end{figure}
\par  
For $m\geq 2$, we have the following observations. 
\begin{obs}\label{prop_N_Mc}
The degree of $N_f$ is $m+n+1$. It is easy to observe the following.
	\begin{enumerate}
		\item Fixed points: The solutions of $z^{m+n}=1$ are superattracting fixed points of $N_f$. In fact, these are simple critical points of $N_f$. In particular, $1$ is a superattracting fixed point for $N_f$ for each $m,n$. The origin and the point at $\infty$ are repelling fixed points of $N_f$ with multipliers $\frac{n+1}{n}$ and $\frac{m}{m-1}$, respectively.
		\item Poles:  Each solution of $z^{m+n}=-\frac{n}{m}$ is a pole of $N_f$. There is a negative pole if $m+n$ is odd (see Figure~(\ref{Plot_N_f}(a))). As $\infty$ is in $\mathcal{J}(N_f)$, each pole is in $\mathcal{J}(N_f)$.
		\item Free critical points: We call a critical point of $N_f$ free if it is not a root of $f$. These are precisely the solutions of  $m(m-1) z^{m+n}-n(n+1)=0$, i.e., $z^{m+n}=\frac{n(n+1)}{m(m-1)}$.
		\item  On the real axis: 
		It follows from Equation~(\ref{N1}) that $N_f(x)-x=-\frac{x\left(x^{m+n}-1\right)}{m x^{m+n}+n}$. Therefore, we have $N_f(x)>x$ for $x \in (0,1)$ and $N_f(x)<x$ for $x \in (1,\infty)$ (see Figure (\ref{Plot_N_f})).
		\item \label{obs-symmetry} Symmetry in dynamics: As $N_f(\lambda z)=\lambda N_f(z)$ for $\lambda^{m+n}=1$, the set of rotations $\left\{z \mapsto \lambda z: \lambda^{m+n}=1\right\}$ is contained in $  \Sigma N_f$, by Lemma~\ref{symmetry}. This also gives that a Fatou component of $N_f$ is mapped onto a Fatou component under each rotation $z \mapsto \lambda z$ with $\lambda^{m+n}=1$.	\end{enumerate}
\end{obs} 
We are now in a position to present the proof of Theorem \ref{New_Mc}.
\begin{proof}[Proof of Theorem \ref{New_Mc}]
	As $N_f$ is conjugate to the Newton's method applied to a polynomial whenever $m=1$, the Julia set $\mathcal{J}(N_f)$ is connected. Assume that $m\geq 2$. Let the attracting fixed points, i.e., the solutions of the equation $z^{m+n}=1$, be denoted by $z_i$ where $i=1,2, \ldots, m+n$ and we assume $z_1=1$. As $\frac{n(n+1)}{m(m-1)}>0$, there is a positive $c$ such that  $c^{m+n} =\frac{n(n+1)}{m(m-1)} $. This $c$ is a free critical point of $N_f$. 
	There are three cases depending on the values of $m+n$.
	
	\par
 Case I: Let $m>n+1$. Then $c<1$. The map $N_f$ is strictly increasing in $[1,\infty)$, and from Observation \ref{prop_N_Mc}(4), we get that $\lim\limits_{n\to \infty} N_f^n(x)=1$ for all $x\in [1,\infty)$. Thus, $[1, \infty)\subset \mathcal{A}_1$, where $\mathcal{A}_1$ denotes the immediate attracting basin of $1$. Hence, $\mathcal{A}_1$ is unbounded. Again, the symmetries in $\mathcal{J}(N_f)$ (see Observation \ref{prop_N_Mc}(5)) give that all the immediate basins are unbounded. Now 
	\begin{eqnarray}
	N_f(c)  =c\left(1+\frac{m-n-1}{m n}\right). \label{cr}
	\end{eqnarray}
	Thus, if $m > n+1$, then $\frac{m-n-1}{m n}> 0$, i.e.,
	$c^*=N_f(c) > c$. From Equation ~(\ref{N'}), we have
	\begin{equation}\label{N'c}
	N_f^{\prime}(z)=\frac{m(m-1)\left(z^{m+n}-1\right)\left(z^{m+n}-c^{m+n}\right)}{\left(m z^{m+n}+n\right)^2},
	\end{equation}
	and 
	$$
	N_f'(x)\left\{\begin{array}{ccc}
	>0 & \text { whenever }& 0<x<c \\
	<0 & \text { whenever } & c<x<1 \\
	>0 & \text { whenever } & x>1.
	\end{array}\right.
	$$
	\begin{figure}[h!]
		\begin{subfigure}{.5\textwidth}
			\centering
			\includegraphics[width=1.0\linewidth]{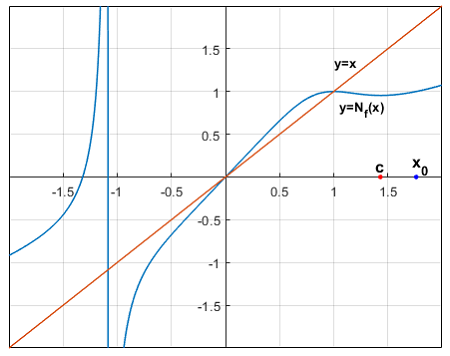}
			\caption{$m=2,~n=3$}
		\end{subfigure}
		\begin{subfigure}{.5\textwidth}
			\centering
			\includegraphics[width=1.0\linewidth]{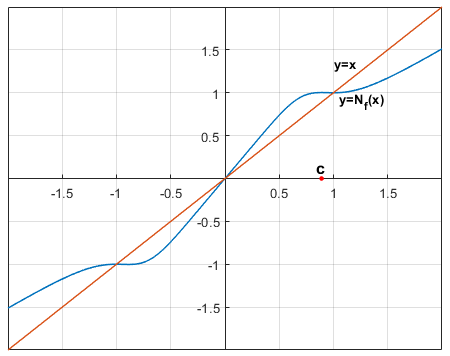}
			\caption{$m=4,~n=2$}
		\end{subfigure}
		\caption{The graph of $N_f (z)=\frac{z((m-1)z^{m+n}+ n+1)}{m z^{m+n}+n}$}
		\label{Plot_N_f}
	\end{figure}
	In the interval $(c, 1)$, $N_f$ is decreasing (see Figure (\ref{Plot_N_f}(b))), which gives that $N_f((c, 1))=(1, c^*)$. This implies that $c \in \mathcal{A}_1$.
	\par 
	The map $N_f$ is increasing in $(0,c)$, giving that $\{N_f ^{n}(x)\}_{n>0}$ is an increasing sequence for each $0< x< c$. If this remains bounded above by $c$, then it must converge and converge to a fixed point lying in $(0,c)$. However, this is not possible as there is no fixed point in $(0,c)$. Therefore, for each $0<x<c$, there is a natural number $n_x$ such that $N_f ^{n_x}(x)>c$. In other words,  $(0,c )\subset \mathcal{A}_1$. Therefore, the positive real axis is contained in $\mathcal{A}_1$. The point $0$ is a repelling fixed point and is in $ \partial \mathcal{A}_1$.
	\par  It follows from Observation~\ref{prop_N_Mc}(5) that each root of $z^{m+n} =c^{m+n}$ is contained in an immediate basin corresponding to an $(m+n)$-th root of unity. Thus, all the free critical points of $N_f$ are in $\bigcup_{i=1}^{m+n}\mathcal{A}_{z_i}$. 
	
	\par As all the immediate basins are unbounded, $0$ cannot be in a bounded component of the Julia set. Thus, both the repelling fixed points of $N_f$ are on the same Julia component. By Theorem \ref{shishikura}, we conclude that the Julia set of $N_f$ is connected.\\
 Case II: 
	If $m=n+1$, then $N_f'(z)=\frac{n(n+1)(z^{2 n+1}-1)^2}{((n+1) z^{2n+1}+n)^2}$ and there are no free critical points. It follows from Theorem~9.3, \cite{Milnor_book} that $\mathcal{J}(N_f)$ is connected.
	\\ {Case III:}
	Now consider $m < n+1$. Then there is a positive free critical point, say $c$ and $c>1$. It follows from Equation~(\ref{N'c}) that
	$$
	N_f'(x)\left\{\begin{array}{ccc}
	>0 & \text { whenever }& 0<x<1 \\
	<0 & \text { whenever } & 1<x<c \\
	>0 & \text { whenever } & x>c.
	\end{array}\right.
	$$
	
	Since  $N_f$ is increasing in $(0,1)$ (see Figure (\ref{Plot_N_f}(a))), $\{N_f ^{n}(x)\}_{n>0}$ is an increasing sequence that is bounded above by $1$ for each $x \in (0,1)$.  Therefore,
	$ \lim\limits_{n \rightarrow \infty} N_f^n(x)=1$ whenever $x \in(0,1)$.
	Now $N_f (c)=c^* <1$ and there is an $x_0 >1$ such that $(1, x_0)$ is mapped onto $(c^*, 1)$. Thus, the interval $(0, x_0) \subset \mathcal{A}_1$. For each $x >x_0$, it is seen that $\{N_f ^{n}(x)\}_{n>0}$ is a strictly decreasing sequence that is bounded below by $1$. It must converge and converge to a fixed point. Since $1$ is the only fixed point on the positive real line, $ \lim\limits_{n \rightarrow \infty} N_f^n(x)=1$ whenever $x > x_0$. The immediate basin $\mathcal{A}_1$ contains the positive free critical point $c$. The rest follows in a similar way as in Case I of this proof.
\end{proof}
Figure (\ref{J_set_N_f}) illustrates the Fatou and Julia set of $N_f$ whenever $m=2, ~n=3$ (Figure (\ref{J_set_N_f}(a))) and $m=4, ~n=2$ (Figure (\ref{J_set_N_f}(b))). The different colors represent different basins. The boundary of any two different colors is in the Julia set.
\begin{figure}[h!]
	\begin{subfigure}{.5\textwidth}
		\centering
		\includegraphics[width=0.98\linewidth]{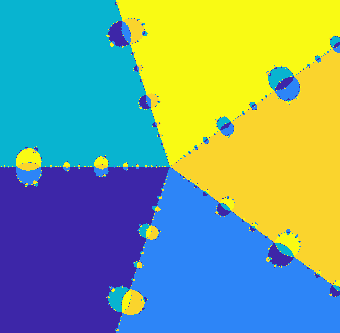}
		\caption{$m=2,~n=3$}
	\end{subfigure}
	\begin{subfigure}{.5\textwidth}
		\centering
		\includegraphics[width=0.98\linewidth]{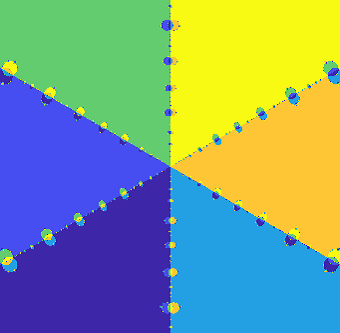}
		\caption{$m=4,~n=2$}
	\end{subfigure}
	\caption{The Fatou set of $N_f (z)=\frac{z((m-1)z^{m+n}+ n+1)}{m z^{m+n}+n}$}
	\label{J_set_N_f}
\end{figure}
Therefore, $\mathcal{F}\left(N_f\right)=\bigcup_{i=1}^{m+n} B_i$ where $B_i$ is the basin of $z_i$.
\begin{rem}
	For $m=1$, we denote the Newton map as $N_{1,n}$. 
	\begin{enumerate}
		\item Although the point at $\infty$ is not a fixed point of $N_{1,n}$, it is a pre-periodic point and is in the Julia set of $N_{1,n}$ as $N_{1,n}(\infty)=0$ and $0$ is a repelling fixed point.
		\item If $n=1$, then  $N_{1,1}$ is conjugate to   $N_{z^2-1}$, whose Julia set is the imaginary axis. Since the conjugating map here is $z \mapsto \frac{1}{z}$, we have that $\mathcal{J}(N_{1,1})$ is the imaginary axis.
	\end{enumerate} 
\label{newton-line}
\end{rem}
\begin{rem}
	It follows from the proof of Theorem \ref{New_Mc} that $N_f$ is a geometrically finite map with connected Julia set. Thus, $\mathcal{J}(N_f)$ is locally connected. Moreover, for $m\geq 2$, $\infty$ is a repelling fixed point of $N_f$, every immediate basin is unbounded and $\infty$ is accessible from each of the immediate basins. In fact, using the same argument used in the proof of Proposition 6, \cite{HSS2001}, it can be shown that there is exactly one access to $\infty$ in each immediate basin. Thus, by Lemma 3.3, \cite{Sym_dyn}, we get that all Fatou components other than the immediate basins are bounded.
\label{bounded-fc}
\end{rem} 
We need a lemma to prove Theorem~\ref{McM_sym}.
 \begin{lem}\label{sym_conjugate}
     Suppose $R_1$ and $R_2$ are two rational maps such that for an affine map $T$, $R_2=T^{-1}\circ R_1\circ T$. Then $\Sigma R_2=\{T^{-1} \circ \phi \circ T: \phi \in \Sigma R_1 \}$.
\end{lem}

\begin{proof}[Proof of Theorem \ref{McM_sym}]
	For  $f_\lambda(z)=z^m-\frac{\lambda}{z^n},	
	\lambda \neq 0$ and $f(z)=\frac{z^{m+n}-1}{z^n}, m+n>2$, we have
	$N_f=T^{-1}\circ N_{f_\lambda}\circ T$ (by Lemma \ref{Conj_N_Mc}) where $T(z)=(\lambda^{\frac{1}{m+n}})z$ . Using Lemma \ref{sym_conjugate}, we have $\Sigma N_f=\{T^{-1} \circ \phi \circ T: \phi \in \Sigma N_{f_\lambda} \}$. Indeed, it can be seen that  $\Sigma N_f= \Sigma N_{f_\lambda} $.

	As $m+n>2$, $N_f$ has at least three superattracting fixed points. Therefore, $\mathcal{J}(N_f)$ cannot be a line. The point at $\infty$ is either a pre-periodic point ($m=1$) or a fixed point ($m>1$) of $N_f$. By a result of Boyd (Theorem 1, \cite{Boyd2000}), we get that $\mathcal{J}(N_f)$ is not invariant under any translation. By Observation \ref{prop_N_Mc}(5), we have $\left\{z \mapsto \lambda z: \lambda^{m+n}=1\right\} \subseteq \Sigma N_f$.  Therefore, every element of $\Sigma N_f$ is a rotation about the origin. It only remains to be shown that each such rotation is of order $m+n$. 
	\par 
	For $m\geq 2$, $\infty$ is a repelling fixed point. Every $\sigma \in \Sigma N_f$ takes an unbounded Fatou component to an unbounded Fatou component of $N_f$, and the immediate basins are the only unbounded components of $\mathcal{F}(N_f)$ (see Remark \ref{bounded-fc}). These are preserved by the rotations of order $m+n$. Therefore, the order of $\sigma$ divides $(m+n)$, and hence $ \Sigma N_{f}=\left\{z \mapsto \lambda z: \lambda^{m+n}=1\right\}$.
	\par 
	If $m=1$, then via inversion $N_{1,n}$ is conjugate to $N_p$, where $p(z)=z^{n+1}-1$. Note that $N_p$ is a geometrically finite map with connected Julia set. Also, the point at $\infty$ is a repelling fixed point of $N_p$, and the immediate basins of attraction of $N_p$ corresponding to the roots of $p$ are the only unbounded Fatou components (see Figure (\ref{m1n2}(b))). Using the same argument used in the previous paragraph, we have $ \Sigma N_{p}=\left\{z \mapsto \lambda z: \lambda^{n+1}=1\right\}$. As $N_{1,n}(z)=\frac{1}{N_p(\frac{1}{z})}$, we get $ \Sigma N_{1,n}=\left\{z \mapsto \lambda z: \lambda^{n+1}=1\right\}$. This concludes the proof.
\end{proof}
\begin{rem}
	In Theorem 3.1, \cite{Ferreira2023} and Corollary 3.3,  \cite{Rot2024}, it is proved that if $f$ has no parabolic or rotation domain, then $\Sigma f=\{\sigma: \sigma(z)=\lambda z,~ \lambda^{m+n}=1\}$. Under this hypothesis, from Theorem \ref{McM_sym}, we get $\Sigma f=\Sigma N_f$.
\end{rem}

\textbf{Acknowledgement:} We express our sincere gratitude to the anonymous referee for carefully reviewing the manuscript. Soumen Pal is supported by Indian Institute of Technology Madras through a Postdoctoral Fellowship. Pooja Phogat is supported by a Senior Research Fellowship provided by the Council of Scientific and Industrial Research, Govt. of India. \\

\section{Declarations}

\subsection{Author Contribution} All authors contributed equally.

\subsection{Funding} Pooja Phogat is funded by a Senior Research Fellowship (Grant number: 09/1059(0031)/2020-EMR-I) provided by the Council of Scientific and Industrial Research, Govt. of India. Funding is not applicable to the other two authors.

\subsection{Conflicts of interest/Competing interests}
Not Applicable.
\subsection{Data Availability statement} Data sharing not applicable to this article as no datasets were generated or analyzed during the current study.

\subsection{Code availability} Not Applicable.\\

\end{document}